\documentclass[11pt, reqno]{amsart}
\usepackage{amsmath, amsthm, amscd, amsfonts, amssymb, graphicx, color, stmaryrd}
\usepackage[all,cmtip]{xy}	
\usepackage[bookmarksnumbered, colorlinks, plainpages]{hyperref}
\usepackage{float}	
\usepackage{tikz-cd}

\tikzset{%
    symbol/.style={%
        ,draw=none
        ,every to/.append style={%
            edge node={node [sloped, allow upside down, auto=false]{$#1$}}}
    }
}

\newcommand{\Rr}{\mathbb{R}}
\newcommand{\Cc}{\mathbb{C}}
\newcommand{\Zz}{\mathbb{Z}}

\newcommand{\CT}{\mathcal{CT}}		
\newcommand{\D}{\mathcal{D}}	

\newcommand{\K}{\mathcal{K}}	
\newcommand{\N}{\mathcal{N}}
\renewcommand{\L}{\mathcal{L}}	
\newcommand{\C}{\mathcal{C}}	
\newcommand{\G}{\mathcal{G}}	
\newcommand{\A}{\mathcal{A}} 	
	
\newcommand{\End}{\mathrm{End}}	
\newcommand{\Hom}{\mathrm{Hom}}
\newcommand{\R}{\mathcal{R}}	

\newcommand{\sigmafull}{\sigma_{\mathrm{full}}}

\newcommand{\Tau}{\mathcal{T}}
\newcommand{\id}{\mathrm{id}}

\newcommand{\tEll}{\mathrm{Ell^{\otimes}}}	
\newcommand{\Ell}{\mathrm{Ell}}
\newcommand{\tSigma}{\Sigma^{\otimes}}

\newcommand{\ex}{\mathrm{ex}}
\newcommand{\B}{\mathcal{B}}
\newcommand{\M}{\mathcal{M}}

\newcommand{\Gop}{\mathrm{\G^{(0)}}} 

\newcommand{\scal}[2]{\langle #1, #2 \rangle}	
\newcommand{\ideal}[1]{\langle #1 \rangle} 	

\newcommand{\op}{\operatorname{Op}} 		
\newcommand{\ind}{\operatorname{ind}}
\newcommand{\potimes}{\hat{\otimes}} 

\newcommand{\pr}{\operatorname{pr}}

\newcommand{\iso}{\xrightarrow{\sim}}	

\newtheorem{Thm}{Theorem}[section]
\newtheorem{Lem}[Thm]{Lemma}
\newtheorem{Prop}[Thm]{Proposition}

\theoremstyle{definition}
\newtheorem{Def}[Thm]{Definition}
\newtheorem{Ex}[Thm]{Example}

\newtheorem{Rem}[Thm]{Remark}

\newtheorem{Not}[Thm]{Notation}

\begin{document}
\setcounter{page}{1}


\title{The Fredholm index for operators of tensor product type}


\author[Karsten Bohlen]{Karsten Bohlen}

\address{$^{1}$ Universit\"at Regensburg, Germany ~\textsf{karsten.bohlen@mathematik.uni-regensburg.de}}






\begin{abstract}
We consider bisingular pseudodifferential operators which are pseudodifferential operators of tensor product type. These operators are defined on the product manifold $M_1 \times M_2$, for closed manifolds $M_1$ and $M_2$. We prove a topological index theorem of product type. In addition, we show that the Fredholm index of elliptic bisingular operators equals the topological index, whenever the operator takes the form of an external tensor product of pseudodifferential operators, up to equivalence. To this end we construct a suitable double deformation groupoid and a Poincar\'e duality type homomorphism.
\end{abstract} \maketitle 



\section{Introduction}

In this note we investigate a class of bisingular pseudodifferential operators. The calculus of bisingular pseudodifferential operators addresses the fact that the tensor product $P_1 \otimes P_2$ in general does not furnish a classical pseudodifferential operator, where $P_1, P_2$ are operators on closed manifolds $M_1$ and $M_2$, respectively. Another example studied in \cite{rodino} is the so-called \emph{vector valued tensor product}, denoted by $P_1 \sharp P_2$. One of the well-known properties of the Fredholm index
is its behavior on such vector valued tensor products of operators: If $P_1, P_2$ are elliptic operators acting on Sobolev spaces as
bounded operators, then the product formula $\ind(P_1 \sharp P_2) = \ind(P_1) \ind(P_2)$ holds, cf. \cite{palais}. 
For an elliptic pseudodifferential operator $P \in \Psi^m(M, E_1, E_2)$ on a closed manifold $M$ and acting on the sections of vector bundles $E_j \to M$, the Atiyah-Singer index theorem yields
a formula of $\ind(P)$, expressing the index in terms of characteristic classes depending on the stable homotopy class of the principal symbol of $P$, \cite{as,asIV}.
A classical and natural question is to study more general tensor products which are Fredholm and to determine how the Fredholm index of such a generalized tensor product behaves, cf. \cite{asIV}.  
The calculus was first described in \cite{rodino}. These operators appear in the consideration of boundary value problems on bicylinders and are also relevant in the study of Toeplitz operators cf. \cite{bgpr, bs, nr}. It is our aim in this note to investigate the index problem, originally stated in \cite{rodino} as well as \cite{asIV}.

The original index problem for smooth closed manifolds can be traced back to I. Gel'fand \cite{gelfand}, who observed the homotopy invariance of the index.
The literature on index theory is vast. Of recent interest has been the index theory of Dirac operators on Lorentzian manifolds, cf. \cite{bs}. We also mention the study of index theory on singular and stratified manifolds, e.g. the work of E. Schrohe, B. W. Schulze \cite{schroheschulze}. As well as the approach to index theory using Lie groupoids \cite{as2011, ds, dl, dln, mn, pz} and the recent solution of the index problem for pseudodifferential operators on Lie groupoids \cite{bl, bohlenschrohe}. Fredholm conditions for operators in various cases are discussed e.g. in \cite{cnq, georgescu, rrs, mantoiu} and the recent work of A. Baldare, R. C\^{o}me, M. Lesch and V. Nistor \cite{bcln} as well as the references therein. 
\newpage

\subsection*{Overview}

\subsubsection*{The problem}

We make use of the notation $\Psi^{m_1, m_2}(M_1 \times M_2)$ to denote the class of bisingular pseudodifferential operators, of multiorder $(m_1, m_2) \in \Rr^2$. For brevity sake, we refer to these operators as \emph{$\otimes$-operators} throughout the remainder of this work. We assume that these are \emph{classical} (sometimes called polyhomogeneous) operators, i.e. they admit asymptotic expansions, cf. \cite{rodino}. 
%

We want to study the following \emph{problem}: Let $P \in \Psi^{m_1, m_2}(M_1 \times M_2, E, F)$, acting on smooth sections of the vector bundles $E, F \to M_1 \times M_2$, be a $\otimes$-operator which is Fredholm as a bounded operator on the scale of Sobolev spaces introduced in \cite{rodino}: 
\[
P \colon H^{s_1, s_2}(M_1 \times M_2, E, F) \to H^{s_1 - m_1, s_2 - m_2}(M_1 \times M_2, E, F), \ (s_1, s_2) \in \Rr^2.
\] 

We study the problem of finding an Atiyah-Singer type index theorem for the Fredholm index $\ind(P)$.
Fredholm operators in the $\otimes$-calculus are precisely the operators which are $\otimes$-elliptic, i.e. operators whose principal symbol mappings are pointwise invertible, cf. Theorem \ref{Thm:telliptic} in section \ref{section:otimes}. The principal symbol mappings take values in families of pseudodifferential operators in the most general case. If $M_1, M_2$ are connected, the structure of the vector bundles $E, F \to M_1 \times M_2$ simplifies. To wit, if $M_1, M_2$ are connected, the pullback bundle $i_{x_1}^{\ast} E$ to $M_2$ along the inclusion $i_{x_1} \colon M_2 \to M_1 \times M_2, \ x_2 \mapsto (x_1, x_2)$ for $x_1 \in M_1$ fixed is a fiber bundle over $M_1$ with typical fiber $i_{x_1}^{\ast} E$ \cite{bs}[Cor. 5.2]. We make a stronger simplifying assumption. The case of general vector bundles will be considered in future work.

\subsubsection*{Assumption} 

We assume the vector bundles $E, F \to M_1 \times M_2$ throughout the document to be trivial along the $M_1, M_2$ fibers, i.e. there are vector bundles $E_i \to M_i, \ F_i \to M_i, \ i = 1,2$ as well as fixed isomorphisms of vector bundles $E \cong E_1 \times E_2, \ F \cong F_1 \times F_2$. Hence $E, F$ are direct products, in the sense that they take the form of the Whitney sums $E = \pr_1^{\ast} E_1 \oplus \pr_2^{\ast} E_2$, $F = \pr_1^{\ast} F_1 \oplus \pr_2^{\ast} F_2$ for $\pr_i \colon M_1 \times M_2 \to M_i, \ i = 1,2$ denoting the projections.

\subsubsection*{Approach}

We consider:
\begin{itemize}
\item A $\otimes$ operator $P$ of order $(m_1, m_2)$ has the operator-valued principal symbol mappings
\begin{align*}
& \sigma_1(P) \colon T^{\ast} M_1 \setminus \{0\} \to \Psi^{m_2}(M_2, E_2, F_2), \\ 
& \sigma_2(P) \colon T^{\ast} M_2 \setminus \{0\} \to \Psi^{m_1}(M_1, E_1, F_1)
\end{align*}
and the scalar principal symbol $\sigma_{m_1, m_2}(P) \colon T^{\ast} M_1 \setminus \{0\} \times T^{\ast} M_2 \setminus \{0\} \to \Cc$. These symbol mappings are compatible, i.e. $\sigma_{m_1, m_2}(P)(x_1, \xi_1, x_2, \xi_2)$ equals $\sigma^{(2)}(\sigma_1(P)(x_1, \xi_2))(x_2, \xi_2) = \sigma^{(1)}(\sigma_2(P)(x_2, \xi_2))(x_1, \xi_1)$, where $\sigma^{(i)}$ denote the respective standard principal symbols, cf. \eqref{comp}.
\item The completion $\overline{\Psi^{0,0}}$ of the $(0,0)$-order operators in the norm of bounded operators $L^2(M_1 \times M_2, E, F) \to L^2(M_1 \times M_2, E, F)$. We have the isomorphism of $C^{\ast}$-algebras \cite[Theorem 3.4]{bohlen2}
\[ 
\overline{\Psi^{0}(M_1, E_1, F_1)} \otimes \overline{\Psi^{0}(M_2, E_2, F_2)} \cong \overline{\Psi^{0,0}(M_1 \times M_2, E, F)}.
\]

\item The continuous extension of the \emph{full symbol} $\sigma_1 \oplus \sigma_2 \colon \overline{\Psi^{0,0}} \to \Sigma^{\otimes}$ which is a homomorphism that identifies the $C^{\ast}$-algebraic quotient $\Sigma^{\otimes} \cong \overline{\Psi^{0,0}} / \K$ by the ideal of compact operators $\K(L^2(M_1 \times M_2, E, F))$ (cf. section \ref{section:otimes}).

\item The inclusion of the continuous functions into the calculus $\mu \colon C(M_1 \times M_2) \to \overline{\Psi^{0,0}}$, via their action as multiplication operators. 

\item The group $\tEll(M_1 \times M_2)$ of multiorder $(0,0)$ operators which are $\otimes$-elliptic, up to stable homotopy (within the class of $\otimes$-elliptic operators) and endowed with the direct sum as composition. 

\item The \emph{relative $K$-theory} $K_0(\overline{\mu})$ of the map $\overline{\mu} = (\sigma_1 \oplus \sigma_2) \circ \mu$ (introduced in Section \ref{section:poincare}).

\item A double deformation groupoid $\Tau \rightrightarrows M_1 \times M_2 \times [0,1]^2$ given as the product groupoid of the tangent groupoid of $M_1$ and the tangent groupoid of $M_2$, cf. \cite[p.100]{connes}.
\end{itemize}

The approach to the problem relies fundamentally on a reduction in $K$-theory which is facilitated by a Poincar\'e duality type homomorphism. By definition of $\Tau$, the $C^{\ast}$-algebra of the double deformation groupoid, $C^{\ast}(\Tau)$, is $K$-equivalent to $C_0(T^{\ast}(M_1 \times M_2))$, cf. Lem. \ref{Lem:SES}. The main argument is based on relative $K$-theory applied to the following diagram: 
\begin{figure}[H]
\begin{tikzcd}[row sep=huge, column sep=huge, text height=1.5ex, text depth=0.25ex]
& C(M_1 \times M_2) \arrow{d}{\mu} \arrow{dr}{\overline{\mu}} & \\
\K \arrow{r} & \overline{\Psi^{0,0}} \arrow{r}{\sigma_1 \oplus \sigma_2} & \tSigma
\end{tikzcd} 
\end{figure}

The relative $K$-theory of the scalar principal symbol $K_{\ast}(\sigma_1 \oplus \sigma_2)$ is isomorphic to $K_{\ast}(\K) \cong \Zz$, since $\sigma_1 \oplus \sigma_2$ is surjective. We then define an excision map $\mu_{\ast} \colon K_0(\overline{\mu}) \to K_0(\sigma_1 \oplus \sigma_2)$ between the associated relative $K$-theory groups. A diagram chase, which uses the structure of the double deformation groupoid, yields the construction of the homomorphism of groups $\Theta$ of the following diagram (cf. section \ref{section:poincare}). Altogether, we obtain the diagram with commuting upper triangles:

\begin{figure}[H]
\begin{tikzcd}[row sep=huge, column sep=huge, text height=1.5ex, text depth=0.25ex]
& \arrow[dl, "\simeq"'] K_0(\sigma_1 \oplus \sigma_2) & \\
\Zz & \arrow{l}{\partial_{nc}} K_1(\Sigma) \arrow[d, dashed]{} \arrow{u}{} \arrow{r}{j_{\ast}} & \arrow{ul}{\mu_{\ast}} K_0(\overline{\mu}) \arrow[d, "\Theta"] \\
& \arrow{ul}{\ind} \tEll(M_1 \times M_2) \arrow[r, swap, "\chi"] \arrow[ur, "\simeq"', "\nu"] & \arrow[bend left=60]{llu}{\ind_a^{\otimes}} K_0(C^{\ast}(\Tau))  
\end{tikzcd} 
\end{figure}

The homomorphism $\chi \colon \tEll(M_1 \times M_2) \to K_0(C^{\ast}(\Tau))$ is the composition $\Theta \circ \nu$. Since the symbol space $\Sigma$ is non-commutative, we cannot expect a "local" index formula, i.e. to be able to express the index of $P$ solely in terms of the stable homotopy data given by the principal symbols. Therefore, the homomorphism $\Theta$ does not in general furnish an isomorphism which would reduce the index computation to the computation of the geometrically defined index map $\ind_a^{\otimes}$.  The topological index theorem of a $\otimes$-elliptic operator can be viewed as the extension to the case of manifolds of the local index formula from \cite{bohlen}[section 5].

\subsubsection*{Result}

The main result Theorem \ref{Thm:PD} therefore states that if $P$ is a $\otimes$-elliptic operator equal up to equivalence in relative $K$-theory to an external tensor product $P_1 \sharp P_2$, then the index equals the topological index associated to the groupoid $\Tau$. To prove the result we need to show that the index can be reduced to the homomorphism $\Theta$. By naturality, the index maps in the previous diagram are reduced to the computation of a geometric $K$-homology group, which turns out to be the ordinary $K$-homology of the Lie groupoid $\Tau$ introduced by Connes. This furnishes a product index formula for operators of this type. Our results are consistent with the $K$-theory computations of \cite{bohlen} and with the topological index theorem of Melrose, Rochon \cite{mr}. The authors of loc. cit. consider a class very similar to $\otimes$-operators. In view of the structure of the Poincar\'e duality type homomorphism, one might venture the guess that the index problem for $\otimes$-operators is always reducible to the index problem of standard H\"ormander pseudodifferential operator on the product manifold. However, as we recognize and the authors of \cite{mr} already noted as well, there exist $K$-theoretic obstructions to such a reduction. 
The reference \cite{nr} contains an expression of the Fredholm index of $\otimes$ Fredholm operators in terms of regularized traces. 


\section{The $\otimes$-calculus}

\label{section:otimes}

\subsection{The $\otimes$-calculus for groupoids}

We extend the calculus of $\otimes$-operators, given in \cite{rodino}, to the case of products of Lie groupoids. The standard H\"ormander pseudodifferential calculus was extended to the case of Lie groupoids by V. Nistor, A. Weinstein and P. Xu in \cite{nwx}. We rely on the constructions made in that article. The adaptations necessary to define the corresponding $\otimes$-calculus for products of Lie groupoids are straightforward.  
Let $\G_1 \rightrightarrows M_1, \ \G_2 \rightrightarrows M_2$ be Lie groupoids over closed manifolds $M_1$ and $M_2$. We
study the product groupoid $\G_1 \times \G_2 \rightrightarrows M_1 \times M_2$ and denote by $\varrho \colon \A(\G_1 \times \G_2) \to T(M_1 \times M_2)$
the corresponding Lie algebroid with anchor map $\varrho$. Note that we have the isomorphism as Lie algebroids $\A(\G_1 \times \G_2) \cong \A(\G_1) \times \A(\G_2)$ \cite{mackenzie}[ch. 4.2. \& Prop. 4.3.10]. We oftentimes use the notation $\A := \A_1 \times \A_2$ for the product Lie algebroid, when the underlying product Lie groupoid $\G := \G_1 \times \G_1$ is understood.

\begin{Def}
The symbol space $S^{m_1, m_2}(\A^{\ast}(\G_1 \times \G_2))$ consists of smooth functions $a \in C^{\infty}(\A^{\ast}(\G_1 \times \G_2))$ such that 
for each $U_1 \times U_2 \subset M_1 \times M_2$ open, trivializing $\A(\G_1 \times \G_2)_{|U_1 \times U_2} \cong (U_1 \times \Rr^{n_1}) \times (U_2 \times \Rr^{n_2})$ we have the uniform estimates:
\[
|\partial_{x_1}^{\alpha_1} \partial_{\xi_1}^{\beta_1} \partial_{x_2}^{\alpha_2} \partial_{\xi_2}^{\beta_2} a(x_1, \xi_1, x_2, \xi_2)| \leq C_{\alpha \beta} \ideal{\xi_1}^{m_1 - |\beta_1|} \ideal{\xi_2}^{m_2 - |\beta_2|}, 
\] 

for each $(x_1, x_2) \in K \subset M_1 \times M_2$ compact. 

\label{Def:tsymbols}
\end{Def}

We are concerned throughout this article with classical symbols, i.e. symbols admitting a bihomogeneous expansion in the sense of \cite{nr, rodino}. By abuse of notation, we keep denoting by $S^{m_1, m_2}$ the sets of classical symbols. We describe next the quantization of a given symbol, by which we mean the rule that associates an operator to a given symbol. To this end we extend the quantization for Lie groupoids as described in \cite{monthubert}. See also \cite{ln, nwx}. The construction depends on choices, but the choices will no longer matter once we work with the completions of the zero order calculus and the corresponding groupoid $C^{\ast}$-algebras.

\begin{Def}
Let $\chi \in C_c^{\infty}(\G_1 \times \G_2)$ be a cutoff function such that $\chi \geq 0$, $\chi^{-1}(0) = M_1 \times M_2$, $\varphi \colon U \to \A := \A_1 \times \A_2$ a $C^{\infty}$-embedding given by the product of tubular neighborhood embeddings $\varphi_1 \times \varphi_2$ such that the following diagram commutes
\[
\xymatrix{
U_1 \times U_2 \ar[d]_{r} \ar[r]^{\varphi_1 \times \varphi_2} & \A_1 \times \A_2 \ar[d]_{\pi} \\
M_1 \times M_2 \ar[r] & M_1 \times M_2
}
\]

and $\varphi(\gamma_1, \gamma_2) = 0 \Leftrightarrow (\gamma_1, \gamma_2) \in M_1 \times M_2$ as well as $T_{(x_1, x_2)} \G_{(x_1, x_2)}$ is isomorphic to the normal bundle $N_{(x_1, x_2)}(\Delta_{M_1} \times \Delta_{M_2})$ via $d \varphi$. 
Define $\mu^{\otimes}$ as the appropriate choice of Haar system on $\G_1 \times \G_2$ as a product system $\mu_{s^{\otimes}(\gamma_1, \gamma_2)}^{\otimes} = \mu_{s(\gamma_1)}^{1} \otimes \mu_{s(\gamma_2)}^{2}, \ (\gamma_1, \gamma_2) \in \G_1 \times \G_2$
for Haar systems $\mu^1$ on $\G_1$ and $\mu^2$ on $\G_2$ respectively. The quantization is defined by $S_{cl}^{m_1, m_2}(\A^{\ast}) \ni a \mapsto \op_{\chi}^{\otimes}(a)(f)(\gamma)$
\begin{align}
&= \int_{\G_{s^{\otimes}(\gamma)}} \int_{\A_{r^{\otimes}(\gamma)}^{\ast}} \chi(\gamma \eta^{-1}) e^{i \scal{\varphi(\gamma \eta^{-1})}{\xi}} a(r^{\otimes}(\gamma), \xi) f(\eta)\, d\xi d\mu_{s^{\otimes}(\gamma)}^{\otimes}(\eta). \label{quant}
\end{align}
\label{Def:quantization}
\end{Def}

The Schwartz kernel of the $\otimes$ pseudodifferential operator $\op_{\chi}(a)$ is a compactly supported distribution conormal to $\Delta_{M_1} \times \Delta_{M_2} \hookrightarrow \G_1 \times \G_2$. Denote by $\Psi^{m_1, m_2}(\G_1 \times \G_2)$ the class of $\otimes$ pseudodifferential operators of order $(m_1, m_2) \in \Rr^2$. These correspond to the conormal distributions, that are endowed with a groupoid convolution product, analogous to the calculus for ordinary Lie groupoids, see \cite{ln, monthubert, nwx} for the details.
We set for the residual classes
\begin{align}
& \Psi^{-\infty, -\infty}(\G_1 \times \G_2) := \bigcap_{(m_1, m_2) \in \Rr^2} \Psi^{m_1, m_2}(\G_1 \times \G_2), \label{residual} \\
& \Psi^{m_1, -\infty}(\G_1 \times \G_2) := \bigcap_{m_2 \in \Rr} \Psi^{m_1, m_2}(\G_1 \times \G_2), \\
& \Psi^{-\infty, m_2}(\G_1 \times \G_2) := \bigcap_{m_1 \in \Rr} \Psi^{m_1, m_2}(\G_1 \times \G_2). 
\end{align}  

\begin{Prop}
Let $a \in S^{m_1, m_2}(\A^{\ast})$ and let $\chi \in C_c^{\infty}(\G_1 \times \G_2), \ \widetilde{\chi} \in C_c^{\infty}(\G_1 \times \G_2)$ be cutoff functions. Then $\op_{\chi}^{\otimes}(a) - \op_{\widetilde{\chi}}^{\otimes}(a) \in \Psi^{-\infty, -\infty}(\G_1 \times \G_2)$.  
\label{Prop:cutoff}
\end{Prop}

\begin{proof}
Write the integral kernel of the difference operator
\[
k^{\otimes}(\gamma) = \int_{\A_{r^{\otimes}(\gamma)}^{\ast}} (\widetilde{\chi}(\gamma) - \chi(\gamma)) e^{i \scal{\varphi(\gamma)}{\xi}} a(r^{\otimes}(\gamma), \xi)\,d\xi. 
\]

The phase function has critical points for $\varphi(\gamma) = 0$. For $\varphi(\gamma) \not= 0$ there is a vector field $L$ such that $L e^{i\scal{\varphi}{-}} = e^{\scal{i \varphi(\gamma)}{-}}$. For any $k$ with $m_1 + m_2 + k < -(n_1 + n_2) +1$ we can write the integral kernel as
\[
k^{\otimes}(\gamma) = \int (\widetilde{\chi}(\gamma) - \chi(\gamma)) e^{i \scal{\varphi(\gamma)}{\xi}} (L^t)^k a(r^{\otimes}(\gamma), \xi)\,d\xi. 
\]

This yields the result.
\end{proof}

Fix the smooth vector bundles $E, F \to M_1 \times M_2$. 
Denote by $\Hom(E, F) \to M_1 \times M_2$ the homomorphism bundles with fibers 
$\Hom(E, F)_{(x_1, x_2)} = \Hom(E_{(x_1, x_2)}, F_{(x_1, x_2)}), \ (x_1, x_2) \in M_1 \times M_2$. Fix the pullback bundles ($r^{\otimes} = r_1 \times r_2$ denoting the range map of the product groupoid $\G_1 \times \G_2$):
\[
(r^{\otimes})^{\ast} E, \ (r^{\otimes})^{\ast} F, \ (r^{\otimes})^{\ast} \Hom(E, F) = \Hom((r^{\otimes})^{\ast} E, (r^{\otimes})^{\ast} F) \to \G.
\]
The space $\Psi^{m_1, m_2}(\G_1 \times \G_2)$ is a local $C_c^{\infty}(\G_1 \times \G_2)$-module and we define:
\begin{align}
& \Psi^{m_1, m_2}(\G_1 \times \G_2, E, F) := \Psi^{m_1, m_2}(\G_1 \times \G_2) \otimes_{C_c^{\infty}(\G_1 \times \G_2)} (r^{\otimes})^{\ast} \Hom(E, F). \label{tPsDo}
\end{align}


The residual operators $\Psi^{-\infty, \infty}(\G_1 \times \G_2)$ are canonically isomorphic to the convolution $\ast$-algebra $C_c^{\infty}(\G_1 \times \G_2)$ by an argument completely analogous to the proof of \cite[Theorem 6]{nwx}. Denote by $L^2(\G_1 \times \G_2)$ the $L^2$-space on the product Lie groupoid $\G_1 \times \G_2$ with respect to the Haar system $\mu^{\otimes}$ and norm denoted by $\|\cdot\|_2$. We refer to \cite{ln}[section 7] for the definition of the Sobolev spaces $H^{s}(\G_i, r_i^{\ast} E_i, r_i^{\ast} F_i), \ i = 1,2$ on Lie groupoids. The $\otimes$ Sobolev spaces are defined as the completed projective tensor product:
\begin{align}
H^{s_1, s_2}(\G_1 \times \G_2, (r^{\otimes})^{\ast} E, (r^{\otimes})^{\ast} F) &:= H^{s_1}(\G_1, r_1^{\ast} E_1, r_1^{\ast} F_1) \potimes H^{s_2}(\G_2, r_2^{\ast} E_2, r_2^{\ast} F_2). \label{space}
\end{align}

We often need to assume $E = F$, whenever compositions need to be defined and the vector bundles have to match. It is our aim to show that the operators in the extended $\otimes$-calculus act boundedly on the appropriate Sobolev spaces. Let $A_i \in \Psi^1(\G_i, r_i^{\ast} E_i) = \Psi^1(\G_i) \otimes_{C_c^{\infty}(\G_i)} \End(E_i)$ be an unbounded elliptic operator in the groupoid calculus, then define via the functional calculus the operators $\Lambda_i := (1 + A_i^{\ast} A_i)^{\frac{1}{2}}$. These furnish isomorphisms for each $s \in \Rr$
\[
\Lambda_i \colon H^s(\G_i, r_i^{\ast} E_i) \iso H^{s-1}(\G_i, r_i^{\ast} E_i). 
\] 

The order reduction operator $\Lambda^{l,k} := \Lambda_1^{l} \otimes \Lambda_2^{k}$ then furnishes isomorphisms on the $\otimes$ Sobolev spaces\footnote{Note the different order convention in \cite{ln}.} \cite[Corollary 4]{ln}
\[
H^{s_1, s_2}(\G_1 \times \G_2, r^{\ast} E) \iso H^{s_1-l, s_2-k}(\G_1 \times \G_2, r^{\ast} E).
\]

\begin{Thm}
Let $P \in \Psi^{m_1, m_2}(\G_1 \times \G_2, r^{\ast} E)$, then $P$ extends to a bounded linear operator:
\[
P \colon H^{s_1, s_2}(\G_1 \times \G_2, (r^{\otimes})^{\ast} E) \to H^{s_1 - m_1, s_2 - m_2}(\G_1 \times \G_2, (r^{\otimes})^{\ast} E).
\]
\label{Thm:Sobolev}
\end{Thm}

\begin{proof}
The application of the order reduction operators leaves, without loss of generality, to prove that a $\otimes$-operator $P$ of order $(0,0)$ is bounded from $L^2(\G_1 \times \G_2, (r^{\otimes})^{\ast} E)$ to $L^2(\G_1 \times \G_2, (r^{\otimes})^{\ast} E)$. By writing the operator $P$ in local groupoid charts as in \cite[Theorem 7]{nwx}, we can argue as in the proof of \cite[Theorem 1.8]{rodino}, to see that it suffices to prove the assertion for the standard pseudodifferential operators on groupoids. The arguments leading up to \cite[Corollary 3]{ln} then suffice to conclude.
\end{proof}

The principal symbol structure of the $\otimes$-calculus can become complicated on a product groupoid. We have to take into account the generalizations of the operator-valued symbols, the scalar symbol and, in the case of groupoids, also the indicial symbol \cite[Theorem 10]{ln}. Let us denote by $\G := \G_1 \times \G_2$ the product groupoid. Assume that the base $M := M_1 \times M_2$ is a compact manifold with boundary or corners. Then the \emph{indicial symbol} 
\[
\R_{H} \colon \Psi^{m_1, m_2}(\G, (r^{\otimes})^{\ast} E) \to \Psi^{m_1, m_2}(\G_{H}, (r^{\otimes})^{\ast} E_{|H})
\]

is defined as the restriction of $\G$-equivariant family 
\begin{align*}
P &= (P_{x_1, x_2})_{(x_1, x_2) \in M} \in \Psi^{m_1, m_2}(\G, (r^{\otimes})^{\ast} E), \\
\R_{H}(P) &= (P_{x_1, x_2})_{(x_1, x_2) \in H}
\end{align*}

where $H$ is some closed boundary hyperface of $M$ that is invariant with regard to $\G$, i.e. $(s^{\otimes})^{-1}(H) = (r^{\otimes})^{-1}(H)$. 

Denote by $p_1 \colon \G_1 \times \G_2 \to \G_1$, $p_2 \colon \G_1 \times \G_2 \to \G_2$ the canonical projections. 
The remaining symbol structure is determined by the three symbol mappings:
\begin{align*}
& \sigma_1 \colon \Psi^{m_1, m_2}(\G, (r^{\otimes})^{\ast} E) \to C^{\infty}(\A_1^{\ast} \setminus \{0\}, \Psi^{m_2}(\G_2, r_2^{\ast} E_2)), \\
& \sigma_1(P)(x_1, \xi_1) \in \Psi^{m_2}(\G_2, ((r^{\otimes})^{\ast} E)_{|\{x_1\} \times M_2}) \cong \Psi^{m_2}(\G_2, r_2^{\ast} E_2), \\
& \sigma_2 \colon \Psi^{m_1, m_2}(\G, (r^{\otimes})^{\ast} E) \to C^{\infty}(\A_2^{\ast} \setminus \{0\}, \Psi^{m_1}(\G_1, r_1^{\ast} E_1)), \\
& \sigma_2(P)(x_2, \xi_2) \in \Psi^{m_1}(\G_1, ((r^{\otimes})^{\ast} E)_{|M_1 \times \{x_2\}}) \cong \Psi^{m_1}(\G_1, r_1^{\ast} E_1), \\
& \sigma_{m_1, m_2} \colon \Psi^{m_1, m_2}(\G, (r^{\otimes})^{\ast} E) \to C^{\infty}(\A_1^{\ast} \setminus \{0\} \times \A_2^{\ast} \setminus \{0\}, (r^{\otimes})^{\ast} \Hom(E, F)).
\end{align*}

Here $\sigma_1, \sigma_2$ are operator valued symbols and $\sigma_{m_1, m_2}$ is referred to as the \emph{scalar symbol}. Denote by $\sigma^{(j)} \colon \Psi^{m_j}(\G_j) \to C^{\infty}(\A_j^{\ast} \setminus \{0\})$ the ordinary principal symbols, $j = 1,2$. Then the compatibility condition holds:
\begin{align}
& \sigma_{m_1, m_2}(P)(x_1, \xi_1, x_2, \xi_2) = \sigma^{(2)}(\sigma_1(P)(x_1, \xi_2))(x_2, \xi_2) = \sigma^{(1)}(\sigma_2(P)(x_2, \xi_2))(x_1, \xi_1). \label{comp}
\end{align}

We refer to $\sigmafull^{\G} := (\sigma_1, \sigma_2, \oplus_{H : \text{hyperface}} \R_H)$ as the \emph{full symbol map}. For the precise definition of the principal symbol mappings, invariantly defined on manifolds, we refer to \cite{rodino}. 




\subsection{The $\otimes$-calculus on manifolds}

The $\otimes$-calculus was defined by L. Rodino in 1975, \cite{rodino}. We proceed to recall some of the notation and fundamental properties of this calculus, that are used throughout this work. Note that the previously defined version of the $\otimes$-calculus over products of Lie groupoids specializes to Rodino's calculus in the following way. The product of closed smooth manifolds $M \times M$ is in particular a Lie groupoid with the groupoid structure of a pair groupoid, i.e. $(x_1, x_2) \cdot (x_2, x_3) = (x_1, x_3)$ the multiplication and $(x_1, x_2)^{-1} = (x_2, x_1)$ the inverse for $(x_1, x_2), \ (x_2, x_3) \in M \times M$. Therefore, if we set $\G_1 = M_1 \times M_1$ and $\G_2 = M_2 \times M_2$ where $M_1, \ M_2$ are closed manifolds, the resulting calculus $\Psi^{m_1, m_2}(\G_1 \times \G_2)$ specializes to Rodino's calculus $\Psi^{m_1, m_2}(M_1 \times M_2)$. Recall the definition of ellipticity for $\otimes$-operators.

\begin{Def}
Let $P \in \Psi^{m_1, m_2}(M_1 \times M_2,  E, F)$, then $P$ is called $\otimes$-elliptic if $\sigma_j(P)$ are pointwise invertible for $j = 1,2$ and $\sigma_{m_1, m_2}(P)$ is pointwise invertible. 
\label{Def:telliptic}
\end{Def}

We note the following property \cite{bs}[Thm. 2.4]:
\begin{Thm}
The operator $P \in \Psi^{m_1, m_2}(M_1 \times M_2, E, F)$ is $\otimes$-elliptic if and only if it is Fredholm as an operator $P \colon H^{s_1, s_2}(M_1 \times M_2, E, F) \to H^{s_1 - m_1, s_2 - m_2}(M_1 \times M_2, E, F)$. 
\label{Thm:telliptic}
\end{Thm}

Throughout this work we only use the direction: $\otimes$-ellipticity implies the boundedness property on Sobolev spaces. Denote by $\K := \K(L^2(M_1 \times M_2, E, F))$ the compact operators. We have the short exact sequence (cf. \cite{bohlen}):
\begin{align}
\xymatrix{
& \K \ar@{>->}[r] & \overline{\Psi^{0,0}}(M_1 \times M_2, E, F) \ar@{->>}[r]^-{\sigma_1 \oplus \sigma_2} & \Sigma^{\otimes}. 
} \label{SES}
\end{align} 

Where by $\overline{\Psi^{0,0}}$ we denote the completion of the algebra of $0$-order operators in $\L(L^2(M_1 \times M_2, E, F))$ operator norm and $\Sigma^{\otimes} = \overline{\Psi^{0,0}} / \K$ is the completed symbol space. Denote by 
\begin{align*}
& \sigma^{M_1} \colon \overline{\Psi^{0}}(M_1, E_1, F_1) \to C(S^{\ast} M_1), \\
& \sigma^{M_2} \colon \overline{\Psi^{0}}(M_2, E_2, F_2) \to C(S^{\ast} M_2)
\end{align*}
the corresponding continuous extensions of the principal symbol homomorphisms for the standard pseudodifferential calculi on $M_1$ and $M_2$ respectively. By an abuse of notation we continue to denote the continuous extensions to the $C^{\ast}$-closures of the operator algebras of the operator valued symbol mappings by $\sigma_1(P)$ and $\sigma_2(P)$. These continuous extensions identify with $\widetilde{\sigma}_1 := \sigma_1(P) = \id_{C(S^{\ast} M_2)} \otimes \sigma^{M_1}$ and $\widetilde{\sigma}_2 := \sigma_2(P) = \sigma^{M_2} \otimes \id_{C(S^{\ast} M_1)}$. The algebra $\Sigma^{\otimes}$ is $\ast$-isomorphic to the pullback \cite{bohlen}[Thm. 3.5]:
\[
\xymatrix{
\Sigma^{\otimes} \ar[d]_{\pi_2} \ar[r]^-{\pi_1} & C(S^{\ast} M_1, \overline{\Psi^{0}}(M_2, E_2, F_2)) \ar[d]_{\widetilde{\sigma}_2} \\
C(S^{\ast} M_2, \overline{\Psi^{0}}(M_1, E_1, F_1)) \ar[r]^{\widetilde{\sigma}_1} & C(S^{\ast} M_1 \times S^{\ast} M_2)
}
\]

\section{The $\otimes$-index}

\label{section:poincare}

\subsection{$\otimes$-Elliptic operators}

In this section we introduce the group of stable homotopy classes of $\otimes$-elliptic operators. In addition, we recall some properties of the index mapping defined on this group.

Let $P \colon H^{s_1, s_2}(M_1 \times M_2, E, F) \to H^{s_1 - m_1, s_2 - m_2}(M_1 \times M_2, E, F)$ acting on the vector bundles $E$ and $F$ be given as an $\otimes$-elliptic operator. Then we note that the Fredholm index is independent of $(s_1, s_2) \in \Rr^2$, cf. \cite{rodino}. Therefore it suffices to consider operators acting as
\[
P \colon H^{m_1, m_2}(M_1 \times M_2, E, F) \to L^2(M_1 \times M_2, E, F).
\]

In addition, by using the bounded transform, we reduce the index computation further to the order $(0, 0)$-case, cf. \cite{bohlen2, bs, rodino}.
 
\begin{Def}
Two $\otimes$ operators $P \colon H^{0,0}(M_1 \times M_2, E) \to H^{0,0}(M_1 \times M_2, F)$, $\widetilde{P} \colon H^{0,0}(M_1 \times M_2, \widetilde{E}) \to H^{0,0}(M_1 \times M_2, \widetilde{F})$ are \emph{stably homotopic} if there exists a continuous homotopy of $\otimes$-elliptic operators
\[
P \oplus 1_{E_0} \sim f^{\ast}(\widetilde{P} \oplus 1_{F_0}) e^{\ast}
\]

where $E_0, F_0 \to M_1 \times M_2$ are vector bundles and $e \colon E \oplus E_0 \to \widetilde{E} \oplus F_0, \ f \colon \widetilde{F} \oplus F_0 \to F \oplus E_0$ are vector bundle isomorphisms.
\label{Def:stablehomotopy}
\end{Def} 

\begin{Rem}
Stable homotopy is an equivalence relation on the set of all $\otimes$-ellipic operators acting on sections of vector bundles. By $\tEll(M_1 \times M_2)$ we denote the set of all $(0,0)$-order $\otimes$-elliptic operators modulo stable homotopies. This set is a group with respect to the direct sum of $\otimes$-elliptic operators. The inverse element corresponds to an an inverse modulo compact operators and the unit is the equivalence class of the identity operator. 
\label{Rem:stablehomotopy}
\end{Rem}

The Fredholm index $\ind$ yields a homomorphism of groups $\ind \colon \tEll(M_1 \times M_2) \to \Zz$, by homotopy invariance. We consider the commutative diagram whose lower part is a short exact sequence:
\begin{figure}[H]
\begin{tikzcd}[row sep=huge, column sep=huge, text height=1.5ex, text depth=0.25ex]
& C(M_1 \times M_2) \arrow{d}{\mu} \arrow{dr}{\overline{\mu}} & \\
\K \arrow{r} & \overline{\Psi^{0,0}} \arrow{r}{\sigma_1 \oplus \sigma_2} & \tSigma
\end{tikzcd} 
\end{figure}

where $\mu$ is the inclusion given by the action as multiplication operators and $\overline{\mu}$ the composition with the quotient map. 

\begin{Def}
Let $f \colon A \to B$ be a $\ast$-homomorphism for $C^{\ast}$-algebras $A$ and $B$. Then $K(f)$ are homotopy class of triples $(p, q, z)$ with $p, q$ projections in $M_{\infty}(A)$ and $z$ invertible in $M_{\infty}(B)$ such that $z f(p) z^{-1} = f(q)$. 
\label{Def:relKthy}
\end{Def}

\begin{Rem}
We have a long exact sequence:
\[
\xymatrix{
\cdots \ar[r] K_{0}(B \otimes C_0(0,1)) \ar[r] & K(f) \ar[r] & K_{0}(A) \ar[r]^{f_{\ast}} & K_{0}(B) \ar[r] & \cdots
}
\]

Given the diagram
\[
\xymatrix{
A \ar[d]_{f} \ar[r]^{\varphi} & B \ar[d]_{g} \\
\widetilde{A} \ar[r]^{\widetilde{\varphi}} & \widetilde{B}
}
\]

there is an associated group homomorphism $(f, g)_{\ast} \colon K(\varphi) \to K(\widetilde{\varphi})$. In addition, $(f, g)_{\ast}$ is an isomorphism, if $f_{\ast}, g_{\ast}$ are $K$-theory isomorphisms. This follows from the five Lemma applied to the long exact sequence. The mapping cone of $\varphi$ is defined as:
\[
\C_{\varphi} := \{(a, F) \in A \oplus C_0((0,1], B) : F(1) = \varphi(a)\}. 
\]

We always have $K(\varphi) \cong K_0(C_{\varphi})$ where $C_{\varphi}$ is the mapping cone of $\varphi$ and excision isomorphism $\ex \colon K(\varphi) \iso K_0(\ker \varphi)$, when $\varphi$ is onto. For an alternative but equivalent definition of $K(\varphi)$, via projective $C^{\ast}$-modules, see also \cite{as2011,blackadar,karoubi} and \cite{bl}[Section 2]. A notable special case is given by the relative $K$-theory $K(X, Y)$ for a CW-pair $(X, Y)$. The generators in this case take the form $(E, F, \alpha)$, where $E, F$ are vector bundles over $X$ and $\alpha \colon E_{|Y} \iso F_{|Y}$ is an isomorphism of the restricted vector bundles. The difference element $d(E, F, \alpha) \in K(X, Y)$ is constructed in \cite{palais}[II. 3] to which we refer for further details. 
\label{Rem:functoriality}
\end{Rem}

By the previous remark we have $K_{\ast}(\mu) \cong K_{\ast}(\C_{\overline{\mu}})$. Here $\C_{\overline{\mu}}$ denotes the mapping cone $C^{\ast}$-algebra of $\overline{\mu}$ given by
\[
\C_{\overline{\mu}} := \{(f, F) \in C(M_1 \times M_2) \oplus C_0((0, 1], \tSigma) : F(0) = \overline{\mu}(f)\}. 
\]

\begin{Thm}
There is an isomorphism of groups $\tEll(M_1 \times M_2) \cong K_0(\C_{\overline{\mu}})$. 
\label{Thm:diffconstr}
\end{Thm}

\begin{proof}
This follows from \cite{savin}[Thm. 4] applied in our case to the data $A_0 := C(M_1 \times M_2), \ A := \overline{\Psi^{0,0}}(M_1 \times M_2) \subset \L(L^2(M_1 \times M_2))$ and mapping cone $\C_{\overline{\mu}}$. 
\end{proof}

In addition we have the so-called mapping cone exact sequence:
\begin{align}
\xymatrix{
S \tSigma \ar@{>->}[r]^{j} & \C_{\overline{\mu}} \ar@{->>}[r]^-{q} & C(M_1 \times M_2) 
} \label{mcses}
\end{align}

Apply the six term exact sequence in $K$-theory to \eqref{SES} and \eqref{mcses} which furnish appropriate connecting maps / boundary maps in $K$-theory. We denote by $\partial_{nc} \colon K_1(\tSigma) \to K_0(\K) \cong \Zz$ the boundary mapping obtained from \eqref{SES}. We have for any $\otimes$-elliptic $P \in \Psi^{0,0}$
\begin{align}
\ind(P) = \partial_{nc}[(\sigma_1 \oplus \sigma_2)(P)]_1
\end{align}

cf. \cite{bohlen}. For the mapping cone exact sequence we get the six term exact sequence:
\[
\xymatrix{
K_1(\tSigma) \ar[r]^{j_{\ast}} & K_0(\C_{\overline{\mu}}) \ar[r]^-{q_{\ast}} & K^0(M_1 \times M_2) \ar[d]_{\epsilon} \\
K^1(M_1 \times M_2) \ar[u]^{\delta} & \ar[l]_-{q_{\ast}} K_1(\C_{\overline{\mu}}) & \ar[l]_{j_{\ast}} K_0(\tSigma). 
}
\]

In particular we have the excision map $j_{\ast} \colon K_{\ast}(\tSigma) \to K_{\ast + 1}(\C_{\overline{\mu}}) \cong K_{\ast + 1}(\overline{\mu})$.

\subsection{Double deformation groupoid}

Given the smooth manifold $M$, A. Connes (cf. \cite[p.100]{connes}) defined the so-called \emph{tangent groupoid}:
\[
\G_{M}^t := TM \times \{0\} \cup M \times M \times (0,1] \rightrightarrows M \times [0,1]
\]

with structure:
\[
\xymatrix{
(\G_M^t)^{(2)} \ar@{->>}[r]^-{m} & \G_M^t \ar@{>->>}[r]^{i} & \G_M^t \ar@{->>}@< 2pt>[r]^-{r,s} \ar@{->>}@<-2pt>[r]^{} & M \times [0,1] \ar@{>->}[r]^-{u} & \G_M^t.
}
\]

We denote by $\pi \colon TM \to M$ the canonical projection and define the structural maps as follows:
\begin{align*}
& s(x, y, t) = y, \ t > 0, \ s(v, 0) = \pi(v), \ t= 0, \\
& r(x, y, t) = x, \ t > 0, \ r(v, 0) = \pi(v), \ t = 0, \\
& m((x, y, t), (y, z, t)) = (x, y, t) \cdot (y, z, t) = (x, z, t), \ t > 0, \\
&(v, 0) \cdot (w, 0) = (v + w, 0), \ t = 0, \\
& i(x, y, t) = (y, x, t), \ t > 0, \ i(v, 0) = -v.
\end{align*}

\begin{Not}
We fix the notation $\G_{M_1}^t \rightrightarrows M_1 \times [0,1]_t$ the tangent groupoid over $M_1$, indexed by $t \in [0,1]$ and $\G_{M_2}^u \rightrightarrows M_2 \times [0,1]_u$ the tangent groupoid over $M_2$, indexed by $u \in [0,1]$. 
\label{Not:Tau}
\end{Not}

\begin{Def}
Given two smooth manifolds $M_1, M_2$ we define the following model groupoid 
\[
\Tau := \G_{M_1}^t \times \G_{M_2}^u \rightrightarrows M_1 \times M_2 \times [0,1] \times [0,1].
\]

with structure:
\[
\xymatrix{
\Tau^{(2)} \ar@{->>}[r]^-{m^{\otimes}} & \Tau \ar@{>->>}[r]^-{i^{\otimes}} & \Tau \ar@{->>}@< 2pt>[r]^-{r^{\otimes},s^{\otimes}} \ar@{->>}@<-2pt>[r]^{} & M_1 \times M_2 \times [0,1]^2 \ar@{>->}[r]^-{u^{\otimes}} & \Tau.
}
\]
\label{Def:Tau}
\end{Def}

\begin{Rem}
In the course of this work we will make use of $C^{\ast}$-algebras of Lie groupoids. For the precise definitions we refer to \cite{lr}[section 8]. Note that Haar systems on Lie groupoids always exist and up to $\ast$-isomorphism there is one $C^{\ast}$- and $C_r^{\ast}$-completion of a given Lie groupoid \cite{lr}[Prop. 3.4]. We recall that a Lie groupoid $\G \rightrightarrows \Gop$ is called \emph{metrically amenable} if the canonical surjective $\ast$-homomorphism $C^{\ast}(\G) \to C_r^{\ast}(\G)$, mapping the full $C^{\ast}$-algebra to the reduced $C^{\ast}$-algebra of $\G$, induced by Definition 2.4. in \cite{cnq}, is also injective. We can show that $C^{\ast}(\Tau)$ is a nuclear $C^{\ast}$-algebra, i.e. there is a unique $C^{\ast}$ tensor product $C^{\ast}(\Tau) \otimes B$ for every $C^{\ast}$-algebra $B$ \cite[15.8.1]{blackadar}. Note that the nuclearity property carries over to all $C^{\ast}$-algebras considered in this section, using the two-out-of-three property \cite[IV.3.1.15]{blackadar}, applied to the short exact sequences induced by the principal symbol mappings.
\label{Rem:Tau}
\end{Rem}

\begin{Thm}
The space $\Tau$ is a $C^{\infty}$-manifold with corners and $\Tau \rightrightarrows M_1 \times M_2 \times [0,1]^2$ has the structure of a metrically amenable Lie groupoid. The $C^{\ast}$-algebra $C^{\ast}(\Tau)$ is nuclear. 
\label{Thm:doubleTG}
\end{Thm}

\begin{proof}
We endow the tangent groupoids with a $C^{\infty}$-structure using a transport of structure argument as described in \cite[p. 103]{connes}. Thereby we obtain Lie groupoids $\G_{M_j}^t$, $j = 1,2$, where the arrow spaces $\G_{M_j}$ are $C^{\infty}$- manifolds with boundary $TM_j = \partial \G_{M_j}^t$ and interior $M_j \times M_j \times (0,1]$. The arrow space $\Tau$ is therefore a $C^{\infty}$-manifold with corners. As a product groupoid of Lie groupoids, $\Tau$ is a Lie groupoid. The groupoid $\Tau$ is measurewise amenable in the sense of  \cite[Proposition 5.3.4]{dr}, since it consists of components that are measurewise amenable. Therefore, the $C^{\ast}$-algebra $C^{\ast}(\Tau)$ is nuclear and metrically amenable \cite[Proposition 2.6]{cnq}, i.e. the full $C^{\ast}$-algebra $C^{\ast}(\Tau)$ is canonically isomorphic to the reduced $C^{\ast}$-algebra $C_r^{\ast}(\Tau)$.
\end{proof}

\begin{Rem}
Given a symbol $a \in S^{m_1, m_2}(T^{\ast}(M_1 \times M_2))$ we define $P_a = (P_{t,u})_{(t, u) \in [0,1]^2}$ in the $\otimes$-calculus $\Psi^{m_1, m_2}(\Tau)$ on the double deformation $\Tau$. In terms of the quantization for pseudodifferential operators of this type as defined by \eqref{quant}, we can write for $t > 0, \ u > 0$ and cutoff function $\chi \in C_c^{\infty}(\Tau)$:
\begin{multline}
(P_{t,u} f)(x_1, x_2, t, u) = \\ 
\int_{M_1} \int_{T_{x_1}^{\ast} M_1} \int_{M_2} \int_{T_{x_2}^{\ast} M_2} e^{\frac{\scal{\exp_{x_1}^{-1}(z_1)}{\xi_1}}{t} + \frac{\scal{\exp_{x_2}^{-1}(z_2)}{\xi_2}}{u}} a(x_1, \xi_1, x_2, \xi_2) \\
\times f(z_1, z_2, x_1, x_2) \chi(z_1, z_2, x_1, x_2) \,\frac{dz_1 dz_2 d\xi_1 d\xi_2}{t^{n_1} u^{n_2}},
\end{multline}

as well as
\begin{multline}
(P_{0,0} f)(x_1, v_1, x_2, v_2; 0, 0) \\ 
= \int_{T_{(x_1, x_2)}(M_1 \times M_2)} \int_{T_{(x_1, x_2)}^{\ast}(M_1 \times M_2)} e^{(v_1 - w_1) \xi_1 + (v_2 - w_2) \xi_2} a(x_1, \xi_1, x_2, \xi_2) \\
\times f(x_1, w_1, x_2, w_2) \chi(x_1, w_1, x_2, w_2) \,dw_1 dw_2 d \xi_1 d\xi_2. 
\end{multline}

Similarly, one can obtain the expressions for $t > 0, \ u = 0$ and $t = 0, \ u > 0$. The full symbol on the calculus $\Psi^{0,0}(\Tau)$ is given as a twice parametrized smooth family of symbol mappings $(\sigma_1^{t,u}, \sigma_2^{t,u}, \R^{t,u})_{(t,u) \in [0,1]^2}$. 
%
\label{Rem:PsDo}
\end{Rem}

\begin{Lem}
Denote by $e_{00} \colon C^{\ast}(\Tau) \to C^{\ast}(T(M_1 \times M_2))$ the restriction homomorphism to the $(0,0)$-component subgroupoid. We have the short exact sequence:
\begin{align}
\xymatrix{
\ker e_{00} \ar@{>->}[r] & C^{\ast}(\Tau) \ar@{->>}[r]^-{e_{00}} & C^{\ast}(T(M_1 \times M_2)) 
} \label{Tses}
\end{align}

The kernel $\ker e_{00}$ is a contractible $C^{\ast}$-algebra.
\label{Lem:SES}
\end{Lem}

\begin{proof}
Note that $\ker e_{00} = C^{\ast}(\Tau_{|[0,1]^2 \setminus \{(0,0)\}})$. We write $A_{u,t}$ for the $C^{\ast}$-algebra $\ker e_{00}$, i.e. it depends on the parameters $(t, u) \in [0,1]^2 \setminus \{(0,0)\}$. Let $f_h \colon A_{u,t} \to A_{u,t}$ be defined as
\[
f_h(g)(\gamma, t, \eta, u) = g(\gamma, t \cdot h, \eta, u \cdot h), \ g \in A_{u,t}.
\]

Then $f_h$ is a $\ast$-homomorphism such that $f_{0} = 0, \ f_1 = \id$, hence $f_h$ is a contracting homotopy. 
\end{proof}

\subsection{Construction of $\Theta$}

We apply Theorem \ref{Thm:diffconstr} to obtain the group isomorphism 
\[
\tEll(M_1 \times M_2) \cong K_0(\C_{\overline{\mu}}).
\] 
Then we proceed to the construction of the Poincar\'e duality type homomorphism 
\[
\Theta \colon K_0(\overline{\mu}) \to K_0(C^{\ast}(\Tau)).
\]

The construction is based on functoriality in relative $K$-theories in conjunction with the structure of the deformation groupoid $C^{\ast}$-algebra $C^{\ast}(\Tau)$. We then show that for $\otimes$-operators this map recovers the Fredholm index, when we restrict to $\otimes$-operators that take the form of external tensor products.

\begin{Def}
We define the \emph{analytic index mapping}:
\begin{align}
\ind_a^{\otimes} &:= (e_{11})_{\ast} \circ (e_{00})_{\ast}^{-1} \colon K_c^0(T^{\ast}(M_1 \times M_2)) \to K_0(\K) \cong \Zz. \label{tinda} 
\end{align}
\label{Def:tinda}
\end{Def}

Assume $E = F$, which is possible by e.g. considering a bundle containing $E \oplus F$. We study the completions of the $\otimes$-pseudodifferential calculus $\Psi^{0,0}(\Tau, (r^{\otimes})^{\ast} E)$ in the $\L(L^2(\Tau, (r^{\otimes})^{\ast} E)$-operator norm according to Theorem \ref{Thm:Sobolev}. The corresponding completions are taken with respect to the appropriate Sobolev norms in the special cases, i.e. we form completions of the calculus on $M_1 \times M_2$ and the calculus on the tangent bundle as a product Lie groupoid $T(M_1 \times M_2) \cong TM_1 \times TM_2$. We fix the notation $\Sigma^{\bullet}$ for the various quotients that are induced by the continuous extensions of the full symbol mapping:
\begin{align*}
\Sigma^c &:= \overline{\Psi^{0,0}}(T(M_1 \times M_2), \pi^{\ast} E) / \ker \sigmafull^{T(M_1 \times M_2)}, \\
\Sigma^{nc} &:= \overline{\Psi^{0,0}}(\Tau, (r^{\otimes})^{\ast} E) / \ker \sigmafull^{\Tau}, \\
\Sigma^{\otimes} &:= \overline{\Psi^{0,0}}(M_1 \times M_2, E) / \ker (\sigma_1 \oplus \sigma_2).
\end{align*}

For brevity we set $\sigma^c := \sigmafull^{T(M_1 \times M_2)}, \sigma^{nc} := \sigmafull^{\Tau}$. We have the following diagram involving the inclusions as actions of continuous functions by multiplication operators:
\begin{figure}[H]
\begin{equation}
\begin{tikzcd}[row sep=huge, column sep=huge, text height=1.5ex, text depth=0.25ex]
C(M_1 \times M_2) \arrow{r}{\mu} & \overline{\Psi^{0,0}}(M_1 \times M_2, E)  \arrow{r}{\sigma_1 \oplus \sigma_2} & \tSigma \\
C(M_1 \times M_2 \times [0,1]^2) \arrow{u}{e_{11}} \arrow{d}{e_{00}} \arrow{r}{\mu^{nc}} & \overline{\Psi^{0,0}}(\Tau, (r^{\otimes})^{\ast} E) \arrow[xshift=.10ex]{u}{e_{11}} \arrow{d}{e_{00}} \arrow{r}{\sigma^{nc}} & \Sigma^{nc} \arrow{u}{e_{11}} \arrow{d}{e_{00}} \\
C(M_1 \times M_2) \arrow{r}{\mu^{c}} & \overline{\Psi^{0,0}}(T(M_1 \times M_2), \pi^{\ast} E) \arrow{r}{\sigma^c} & \Sigma^c
\end{tikzcd}
\label{Dia}
\end{equation}
\end{figure}
In the notation of Remark \ref{Rem:PsDo} $\sigma_1 = \sigma_{1}^{1,1}, \ \sigma_2 = \sigma_2^{1,1}, \ \R^{1,1} = \sigma_{00}$. 
We also set $\overline{\mu} := (\sigma_1 \oplus \sigma_2) \circ \mu, \ \overline{\mu}^{nc} := \sigma^{nc} \circ \mu^{nc}$ and $\overline{\mu}^{c} := \sigma^c \circ \mu^c$ for the induced actions on the quotients. 
We have the following commuting diagrams:
\begin{align}
\xymatrix{
C(M_1 \times M_2) \ar[d]_{\mu} \ar[r]^-{\overline{\mu}} & \tSigma \ar@{==}[d] \\
\overline{\Psi^{0,0}} \ar@{->>}[r]^{\sigma_1 \oplus \sigma_2} & \tSigma
} \label{1}
\end{align}

as well as
\begin{align}
\xymatrix{
C(M_1 \times M_2 \times [0,1]^2) \ar[d]_{e_{11}} \ar[r]^-{\overline{\mu}^{nc}} & \Sigma^{nc} \ar[d]_{e_{11}} \\
C(M_1 \times M_2) \ar[r]^{\overline{\mu}} & \tSigma
}
\label{2}
\end{align}

and
\begin{align}
\xymatrix{
C(M_1 \times M_2 \times [0,1]^2) \ar[d]_{e_{00}} \ar[r]^-{\overline{\mu}^{nc}} & \Sigma^{nc} \ar[d]_{e_{00}} \\
C(M_1 \times M_2) \ar[r]^{\overline{\mu}^{c}} & \Sigma^c.
}
\label{3}
\end{align}

There are induced homomorphisms of groups $\mu_{\ast} \colon K_{\ast}(\overline{\mu}) \to K_{\ast}(\sigma_1 \oplus \sigma_2)$ and likewise $\mu_{\ast}^{nc}$ and $\mu_{\ast}^{c}$. Let us outline the construction of $\mu_{\ast}$ using diagram \eqref{1} for example, see also \cite{pz}[section 10.2]. By surjectivity of $\sigma_1 \oplus \sigma_2$, the inverse of the excision map furnishes the first isomorphism in the following chain of isomorphisms:
\[
K_{\ast}(\sigma_1 \oplus \sigma_2) \cong K_{\ast}(\ker(\sigma_1 \oplus \sigma_2)) \cong K_{\ast}(\K) \cong \Zz. 
\]

On the level of generators, the inverse of the excision map is described as follows. Any class in $K_{\ast}(\sigma_1 \oplus \sigma_2)$ is represented by a relative cycle $[p, q, z]$ with respect to the quotient map $\widetilde{q} \colon \K^{+} \to \Cc$ from the unitalization of $\K$ to $\Cc$. Then $[p, q, z]$ is sent to $[p] - [q]$, which by definition of $K_{\ast}(\widetilde{q})$ is a class in $K_{\ast}(\K)$. Represent a class in $K_{\ast}(\mu)$ via the full symbol of a $\otimes$-elliptic operator $P$, as follows
\begin{align}
[P]_{rel} &= \left[\begin{pmatrix} 1_{E} & 0 \\ 0 & 0 \end{pmatrix}, \begin{pmatrix} 0 & 0 \\ 0 & 1_{F} \end{pmatrix}, \begin{pmatrix} 0 & (\sigma_1 \oplus \sigma_2)(P)^{-1} \\ (\sigma_1 \oplus \sigma_2)(P) & 0 \end{pmatrix}\right]. \label{x}
\end{align}

This class is the pair of $C(M_1 \times M_2)$-modules which are isomorphic via $(\sigma_1 \oplus \sigma_2)(P)$, if viewed as $\Sigma$-modules through $\overline{\mu}$. The invertible lift of 
\[
\begin{pmatrix} 0 & (\sigma_1 \oplus \sigma_2)(P)^{-1} \\
(\sigma_1 \oplus \sigma_2)(P) & 0 \end{pmatrix}
\]

is given by
\[
T = \begin{pmatrix} 1_E - QP & -(1_E - QP) Q + Q \\
P & 1_F - PQ \end{pmatrix}. 
\]

With $Q$ being a parametrix (up to compact errors) of $P$. By the Kuiper's theorem we can take a path $T_t$ through the invertible elements from $1_E \oplus 1_F$ to $T$. This furnishes a homotopy from the relative cycle $(P)_{rel}$ generating the class $[P]_{rel}$ as specified above, to a relative cycle with respect to $\widetilde{q}$, whose image in $K_{\ast}(\K)$ is given by 
\begin{align}
& T \begin{pmatrix} 1_{E} & 0 \\ 0 & 0 \end{pmatrix} T^{-1} - \begin{pmatrix} 0 & 0 \\ 0 & 1_{F} \end{pmatrix} \in M_2(\K). \label{D}
\end{align}

This defines the image of $[P]_{rel}$ under $\mu_{\ast}$. Note in particular that the boundary map $\partial_{nc} \colon K_1(\tSigma) \to K_0(\K) \cong \Zz$ from the six term exact sequence of the short exact sequence \eqref{SES}, sends $[(\sigma_1 \oplus \sigma_2)(P)]_1$ to the class \eqref{D}. 

The vertical maps in diagram \eqref{2} are surjective and the kernels of $e_{11}$ are functions $f$ on $[0,1]_t \times [0,1]_u$ respectively such that $f(1,1) = 0$. These kernels are thereby isomorphic to mapping cones. Hence, they induce isomorphisms $(e_{11})_{\ast}$ in $K$-theory. By Remark \ref{Rem:functoriality}, the induced map in relative $K$-theory $(e_{11}, e_{11})_{\ast} =: (e_{11})_{\ast} \colon K_{\ast}(\overline{\mu}^{nc}) \to K_{\ast}(\overline{\mu})$ is an isomorphism. We denote its inverse by $(\varphi^{0,0})_{\ast} := (e_{11})_{\ast}^{-1} \colon K_{\ast}(\overline{\mu}) \to K_{\ast}(\overline{\mu}^{nc})$.   

Altogether, by an application of functoriality from Remark \ref{Rem:functoriality} to \eqref{Dia}, we obtain the following diagram:
\begin{figure}[H]
\begin{equation}
\begin{tikzcd}[row sep=huge, column sep=huge, text height=1.5ex, text depth=0.25ex]
K_{\ast}(\overline{\mu}^c) \arrow{d}{\mu_{\ast}^c} & \arrow{l}{(e_{00})_{\ast}} K_{\ast}(\overline{\mu}^{nc}) \arrow{d}{\mu_{\ast}^{nc}}  & \arrow{l}{(\varphi^{0,0})_{\ast}} K_{\ast}(\overline{\mu}) \arrow{d}{\mu_{\ast}} \\
K_{\ast}(\sigma^c) \arrow{d}{ex^c} & \arrow{l}{(e_{00})_{\ast}} K_{\ast}(\sigma^{nc}) \arrow{d}{ex^{nc}} & \arrow{l}{(\varphi^{0,0})_{\ast}} K_{\ast}(\sigma) \arrow{d}{ex} \\
K_c^{\ast}(T^{\ast}(M_1 \times M_2)) & \arrow{l}{(e_{00})_{\ast}} K_{\ast}(C^{\ast}(\Tau))  & \arrow{l}{(\varphi^{0,0})_{\ast}} K_{\ast}(\K) \cong \Zz
\end{tikzcd}
\label{Diagram}
\end{equation}
\end{figure}

The homomorphism $\Theta$ is defined as the composition of the homomorphisms going from the upper right corner to the lower left corner of the diagram, i.e.:
\begin{align}
\Theta := \ex^c \circ \mu_{\ast}^c \circ (e_{00})_{\ast} \circ (\varphi^{0,0})_{\ast}. \label{Theta}
\end{align}

\begin{Ex}
The external tensor product of elliptic operators is defined as 
\[
P = P_1 \sharp P_2 = \begin{pmatrix} P_1 \otimes I & -I \otimes P_2^{\ast} \\ I \otimes P_2 & P_1^{\ast} \otimes I \end{pmatrix}, 
\]
In this case the principal symbols take the form 
\begin{align*}
\sigma_1(P)(x_1, \xi_1) &= \begin{pmatrix} \sigma(P_1)(x_1, \xi_1) & -P_2^{\ast} \\
P_2 & \overline{\sigma(P_1)(x_1, \xi_1)} \end{pmatrix}, \\
\sigma_2(P)(x_2, \xi_2) &= \begin{pmatrix} P_1 & -\overline{\sigma(P_2)(x_2, \xi_2)} \\
\sigma(P_2)(x_2, \xi_2) & P_1^{\ast} \end{pmatrix}.
\end{align*}

Systems of the form $P_1 \sharp P_2$ were studied by Atiyah and Singer with $P_1$ and $P_2$ classical pseudodifferential operators on $M_1$ and $M_2$ respectively and of order $> 0$. The external tensor product $P_1 \sharp P_2$ results in an operator which is not classical in the standard pseudodifferential calculus on $M_1 \times M_2$. It is contained in the calculus of classical $\otimes$-operators, \cite{rodino}. If $P_1$, $P_2$ are elliptic of order $m_1 > 0,\ m_2 > 0$, then the standard symbol of $P_1 \sharp P_2$ is a matrix of $2 \times 2$ continuous functions. Then $P_1 \sharp P_2$ can be shown to be Fredholm, via a uniform approximation by classical elliptic pseudodifferential operators, cf. Atiyah-Singer \cite{as}[pp. 512-515] and \cite{hoermander}[Thm. 19.2.7]. In general, by an application of the $\otimes$ calculus the external tensor product $P_1 \sharp P_2$ for order $\leq 0$ operators is a Fredholm operator over $L^2(M_1 \times M_2, E, F)$ with index $\ind(P_1 \sharp P_2) = \ind(P_1) \cdot \ind(P_2)$. 
\label{Rem:lifting}
\end{Ex}

\begin{Thm}
\[
\xymatrix{
& \Zz \\
\Ell^{\otimes}(M_1 \times M_2) \ar[dr]_{\nu} \ar[ur]^{\ind} \ar[r]^-{\chi} & K_0(C^{\ast}(\Tau)) \ar[u]_{\ind_a^{\otimes}} \\
& K_0(\overline{\mu}) \ar[u]^{\Theta}
}
\]

There is a canonical homomorphism $\Theta$ and the non-commutative difference construction group isomorphism $\nu$, where we set $\chi := \Theta \circ \nu$ such that: Let $[P] \in \Ell^{\otimes}(M_1 \times M_2)$ be the class of a $\otimes$-elliptic operator $P$. If $P$ is an external tensor product (up to equivalence in relative $K$-theory), then $\ind_a^{\otimes}$ recovers the Fredholm index of $P$, i.e.:
\[
\ind[P] = (\ind_a^{\otimes} \circ \chi)[P]. 
\]

\label{Thm:PD}
\end{Thm}

\begin{proof}
Let $P_{11} = P \in \Psi^{0,0}(M_1 \times M_2, E)$ be $\otimes$-elliptic and of the form $P_1 \sharp P_2$ for elliptic pseudodifferential operators $P_i \in \Psi^{0}(M_i, E_i), i=1,2$. Let $P_a = (P_{t,u})_{0 \leq t,u \leq 1}$ the family of $\otimes$ operators in the calculus of $\Tau$, depending on the symbol $a$. Fix $(a, a_1, a_2)$ such that $((e_{01})_{\ast} \circ (\varphi^{0,0})_{\ast} \circ \ex \circ \mu_{\ast})[P] = [a_1], \ ((e_{10})_{\ast} \circ (\varphi^{0,0})_{\ast} \circ \ex \circ \mu_{\ast})[P] = [a_2]$, $[a] := (\Theta \circ \nu)[P] \in KK(\Cc, C^{\ast}(T(M_1 \times M_2))$. 
Write $\ind_a^{\otimes} = - \otimes [e_{00}]^{-1} \otimes [e_{11}]$. 

A homotopy construction from \cite{palais}[pp. 220 - 223] yields 
\begin{align}
& [a] = (\Theta \circ v)[P] = [a_1] \otimes_{\Cc} [a_2] \in KK(\Cc, C^{\ast}(T(M_1 \times M_2))). \label{equality}
\end{align}

To make the paper more self-contained we include the details. The construction cited above takes place in (commutative) relative $K$-theory (see also Remark \ref{Rem:functoriality}) and the final equality in $KK$-theory follows by an application of the Atiyah-Singer difference isomorphism between relative $K$-theories and the compactly supported $K$-theory of the tangent bundles \cite{as}. We set throughout $M := M_1 \times M_2$. We introduce the notation $\nu_i$ to denote the classical Atiyah-Singer difference map, associating a relative cycle generating a class in $K(B(M_i), S(M_i))$ to a given pseudodifferential operator on $M_i$, as well as $\nu^c$ will denote the map associating a relative cycle generating a class in $K(B(M), S(M))$ to a pseudodifferential operator on $M$. By an abuse of notation we denote by $\pi$ respectively the projections $TM \to M, \ TM_i \to M_i$ or projections of ball bundles, specified wherever it is not clear from context which is meant. We can assume for the following argument without loss of generality that $P_i$ are elliptic of order $> 0$. Indeed, we can choose $\widetilde{P}_i$ of order $> 0$ such that $\widetilde{D} = \widetilde{P}_1 \sharp \widetilde{P}_2$ and the index of $\widetilde{D}$ agrees with the index of $P_1 \sharp P_2$, by an application of the bounded transform $\widetilde{D} (1 + \widetilde{D}^2)^{-\frac{1}{2}}$ as a $\otimes$ operator. Furthermore, we view the class $[a]$ as the class of the principal symbol of a standard elliptic (H\"ormander) pseudodifferential operator $D \colon C^{\infty}(M, E) \to C^{\infty}(M, F)$ of order $> 0$ that is given by $P_1 \sharp P_2$. This is possible by virtue of the facts stated in Remark \ref{Rem:lifting}, since even though by itself $P_1 \sharp P_2$ may not be a standard pseudodifferential operator, we may choose a $D$ that is elliptic without changing the index, by virtue of \cite{as}[pp. 512-515] and \cite{hoermander}[Thm. 19.2.7]. 

\begin{figure}[H]
\begin{tikzcd}[row sep=huge, column sep=huge, text height=1.5ex, text depth=0.25ex]
& K_0(T^{\ast} M_1) \arrow[r, "\simeq"', "\nu_1"] & K(BM_1, SM_1) \\
K_{0}(C^{\ast}(\Tau)) \arrow[dr, swap, "(e_{01})_{\ast}"] \arrow{ur}{(e_{10})_{\ast}} \arrow[r, "\simeq"', "(e_{00})_{\ast}"] & K(T^{\ast} M) \arrow[r, "\simeq"', "\nu^c"] & K(BM, SM) \\
& K_0(T^{\ast} M_2) \arrow[r, "\simeq"', "\nu_1"] & K(BM_2, SM_2)
\end{tikzcd}
\end{figure}

The (classical) difference construction furnishes the difference class:
\[
\nu^c(D) = d(\pi^{\ast} E, \pi^{\ast} F, \sigma(D)) \in K(B(M), S(M)).
\]

The symbol $\sigma(P_i) \colon \pi^{\ast} E_i \to \pi^{\ast} F_i$ extends by continuity to $BM_i$. Next we construct a homotopy, to show that $\nu(D)$ is homotopic to the product $\nu_1(D_1) \times \nu_2(D_2)$. The bundles $E, F \to M$ and the principal symbol yields a bundle map
\begin{align}
& \sigma(P_1 \sharp P_2) \colon \pi^{\ast} E \to \pi^{\ast} F \label{symboliso}
\end{align}

with $\pi \colon B(M_1) \times B(M_2) \to M$ defined over $B(M_1) \times B(M_2)$. The map \eqref{symboliso} is an isomorphism off the zero cross-sections of $B(M_1) \times B(M_2)$. By definition, 
\begin{align}
& \nu^c(P_1 \sharp P_2) = d(\pi^{\ast} E, \pi^{\ast} F, \sigma(P_1 \sharp P_2)_{|S(M)} \label{diffconstruction}
\end{align}

where $\pi \colon B(M) \to M$. This furnishes an element of $K(B(M), S(M))$. Note that
\[
\partial(B(M_1) \times B(M_2)) = B(M_1) \times S(M_2) \cup S(M_1) \times B(M_2). 
\]

Then by \cite{palais}[II, Thm. 8] together with 
\begin{align}
& \nu_i(P_i) = d(\pi^{\ast} E_i, \pi^{\ast} F_i, \sigma(P_i)_{|S(M_i)} \label{diffconstructionPi}
\end{align}

shows that
\begin{align}
\nu_1(P_1) \times \nu_1(P_2) = d(\pi^{\ast} E, \pi^{\ast} F, \sigma(P_1 \sharp P_2)_{|\partial(B(M_1) \times B(M_2)}. \label{diffproduct}
\end{align}

Both sides of \eqref{diffconstructionPi} lie in $K(B(M_1) \times B(M_2), \partial (B(M_1) \times B(M_2))$. Let 
\[
E(M) := \{(x_1, x_2) \in B(M_1) \times B(M_2) : \|(x_1, x_2)\| \geq 1\}. 
\]

Then $E(M)$ is the part of $B(M_1) \times B(M_2)$ lying between $S(M)$ and $\partial (B(M_1) \times B(M_2))$. Note that $B(M_1 \times M_2) \subseteq B(M_1) \times B(M_2)$. We have the following homotopy commutative diagram of homotopy equivalences:
\[
\xymatrix{
(B(M_1) \times B(M_2), \partial(B(M_1) \times B(M_2)) \ar[d]_{r} \ar[rd]^{i_1} & \\
(B(M), S(M)) \ar[r]^{i_2} & (B(M_1) \times B(M_2), E(M))
}
\]

The homotopy $r$ is defined 
\begin{align*}
& r \colon (B(M_1) \times B(M_2), \partial(B(M_1) \times B(M_2)) \\
&= (B(M_1) \times B(M_2), B(M_1) \times S(M_2) \cup S(M_1) \times B(M_2)) \to (B(M), S(M))
\end{align*}

to equal the identity on $B(M_1) \times B(M_2)$ and the radial retraction onto $S(M)$ on $B(M)$. In addition, $i_1, i_2$ denote inclusions. Since $E(M)$ consists of non-zero vectors, the map $\sigma(P_1 \sharp P_2)_{|E(M)}$ is an isomorphism. Define $\nu^{\ast} \in K(B(M_1) \times B(M_2), E(M))$ by $\nu^{\ast} = d(\pi^{\ast} E, \pi^{\ast} F, \sigma(P_1 \sharp P_2)_{|E(M)}$. Next, we make use of the operation $^{!}$ on $K$-theory, induced by the pullback of vector bundles (cf. \cite{palais}[ch.2]). Then, via the homotopy diagram
\[
r^{!} i_2^{!}(\nu^{\ast}) = i_1^{!}(\nu^{\ast}). 
\]

By \eqref{diffproduct}, 
\[
i_1^{!}(\nu^{\ast}) = \nu_1(P_1) \times \nu_2(P_2)
\]

and by \eqref{diffconstruction} 
\[
i_2^{!}(\nu^{\ast}) = \nu^c(P_1 \sharp P_2).
\]

Thus $r^{!} \nu^c(P_1 \sharp P_2) = \nu_1(P_1) \times \nu_2(P_2)$. Via the isomorphisms induced by the Atiyah-Singer difference construction this furnishes the equality \eqref{equality}.

By virtue of this identity, a computation in $KK$-theory yields
\begin{align*}
&\ind_a^{\otimes}((\Theta \circ \nu)[P]) = [a] \otimes [e_{00}]^{-1} \circ [e_{11}] \\
&= [a_1] \otimes_{\Cc} [a_2] \otimes [e_{00}]^{-1} \otimes [e_{11}] = [P_{11}] \in KK(\Cc, \K) \cong K_0(\K). 
\end{align*}

Here $[P_{11}]$ coincides with $\ind(P_{11})$ under $K_0(\K) \cong \Zz$. Now $P_{11}$ has scalar symbol equal to the leading part of $a$ and every class in $K_c^0(T^{\ast}(M_1 \times M_2))$ can be obtained from such a symbol.
\end{proof}

\section{Topological index theorem}

We define next the $\otimes$-topological index mapping $\ind_{\mathrm{top}}^{\otimes}$ and prove the index theorem, i.e. $\ind_a^{\otimes} = \ind_{\mathrm{top}}^{\otimes}$. 
Denote by $\Tau^{t,u} := \Tau$ the double tangent groupoid defined in the previous section. Fix the embeddings
$i_1 \colon M_1 \hookrightarrow \Rr^{N_1}, \ i_2 \colon M_2 \hookrightarrow \Rr^{N_2}$. We introduce the 
homomorphism $h^{t,u} \colon (\Tau^{t,u}, \cdot) \to (\Rr^{N_1 + N_2}, +)$ which is defined as follows
\begin{align*}
& h^{t,u}(x, y, t, w, 0) = \left(\frac{i_1(x) - i_1(y)}{t}, (di_2)_{\pi_2(w)}(w)\right), \ t > 0, u = 0, \\
& h^{t,u}(v, 0, \widetilde{x}, \widetilde{y}, u) = \left((di_1)_{\pi(v)}(v), \frac{i_2(\widetilde{x}) - i_2(\widetilde{y})}{u}\right), \ t=0, u >0, \\
& h^{t,u}(v,0,w,0) =((di_1)_{\pi(v)}(v), (di_2)_{\pi(w)}(w)), \ t = 0, u=0, \\
& h^{t,u}(x,y,t,\widetilde{x}, \widetilde{y}, u) = \left(\frac{i_1(x) - i_1(y)}{t}, \frac{i_2(\widetilde{x}) - i_2(\widetilde{y})}{u}\right), \ t > 0, u > 0.
\end{align*}

Using $h$ we obtain a right action of $\Tau$ on $\Rr^{N_1 + N_2}$ and the action groupoid 
$\Tau_h := \Rr^{N_1 + N_2} \rtimes_h \Tau \rightrightarrows \Tau^{(0)} \times \Rr^{N_1 + N_2}$ with the following structure:
\begin{align*}
& s_h(v, \gamma) = (v + h(\gamma), s^{\otimes}(\gamma)), \\
& r_h(v, \gamma) = (v, r^{\otimes}(\gamma)), \\
& (v, \gamma) \cdot (v + h(\gamma), \eta) = (v, \gamma \cdot \eta), \\
& u_h(v, x) = (v, x), \ \text{note that} \ h(x) = 0 \ \text{for} \ x \in \Tau^{(0)}, \\
& (v, \gamma)^{-1} = (v + h(x), \gamma^{-1}), \ \text{note that} \ h(\gamma) + h(\gamma^{-1}) = 0. 
\end{align*}

\begin{Lem}
\emph{i)} The action of $\Tau^{t,u}$ on $\Rr^{N_1 + N_2}$, induced by $h$, is free and proper.

\emph{ii)} The orbit space $\Tau^{(0)} \times \Rr^{N_1 + N_2} / \sim$ is homeomorphic to the locally compact space:
\[
\B^{t,u} = ((0,1]_t \times \Rr^{N_1} \cup \N_1 \times \{0\}_t) \times ((0,1]_u \times \Rr^{N_2} \cup \N_2 \times \{0\}_u).
\]

Where by $\N_1 = \N(i_1), \ \N_2 = \N(i_2)$ we denote the normal bundles to $i_1$ and $i_2$ respectively. 
\label{Lem:Tauh}
\end{Lem}

\begin{proof}
\emph{i)} To check that $\Tau_h$ is a free and proper groupoid we need to verify that $\Tau_h$ has isotropy groups which are
quasi-compact and that $(r_h, s_h) \colon \Tau_h \rightrightarrows \Tau^{(0)} \times \Rr^{N_1 + N_2}$ is a closed map, by virtue of \cite{tu}[Prop. 2.14]. Writing $h = h^t \times h^u$, where $h^t, h^u$ are the corresponding homomorphisms $\G_{M_1}^t \to \Rr^{N_1}, \ \G_{M_2}^u \to \Rr^{N_2}$, we see that
\begin{align}
& \Rr^{N_1 + N_2} \rtimes_h \Tau \cong (\Rr^{N_1} \rtimes_{h^t} \G_{M_1}^t) \times (\Rr^{N_2} \rtimes_{h^u} \G_{M_2}^u) \label{Tauh}
\end{align}

as Lie groupoids. By the injectivity of $h^t$ and $h^u$ on the $r$- and $s$-fibers the isotropy groups of $\Tau_h$ consist of only single points, therefore
they are quasi-compact. Secondly, $h^t$ and $h^u$ induce closed range / source maps, cf. \cite{clm}. Thus, since $h = h^t \times h^u$ 
and by \eqref{Tauh} $r_h \oplus s_h$ is closed. 
\emph{ii)} We endow $\B^{t,u}$ with the canonical locally compact topology. 

The computation in \cite{connes} furnishes the orbit spaces $(\G_{M_1}^t)^{(0)} \times \Rr^{N_1} / \sim \ \cong \B^t$
with regard to the $h^t$-induced action and $(\G_{M_2}^u)^{(0)} \times \Rr^{N_2} / \sim \ \cong \B^u$ with regard to the $h^u$-induced action.
This yields the assertion.
\end{proof}

\begin{Rem}
Let us consider the orbit space for the tangent groupoid $\G_{M_1 \times M_2}^t \rightrightarrows M_1 \times M_2 \times [0,1]$ 
with regard to the corresponding homomorphism $h_{M_1 \times M_2}^t$, denoted by $\B_{M_1 \times M_2}^t$.
Note the $K$-equivalences $\B^{t,u} \sim_{K} \N_1 \times \N_2 \sim_{K} \B_{M_1 \times M_2}^t$. 
However, $\B_{M_1 \times M_2}^t$ is not the orbit space of the $\Tau$-action induced by $h^{t,u}$, since the underlying
equivalence relation is different. 
\label{Rem:Tauh}
\end{Rem}

We recall next the construction of the Connes-Thom map from \cite{clm}[Def. 2.9, Prop. 2.10]. For sake of completeness we include full details below.

\begin{Thm}
Given the data $(\Tau, h)$ as specified and assuming that $N_1 + N_2$ is even, then there is a map in $K$-theory: 
$\CT_h^{t,u} \colon K_{\ast}(C^{\ast}(\Tau)) \to K_{\ast}(C^{\ast}(\Tau_h))$ such that

\emph{i)} $\CT_h^{t,u}$ is natural.

\emph{ii)} $\CT_h^{t=0,u=0}$ yields the Thom-isomorphism $\tau \colon K_c^0(T(M_1 \times M_2)) \iso K_c^0(\N_1 \times \N_2)$.

\emph{iii)} The morphism $\CT_h$ coincides with the usual Connes-Thom isomorphism. 
\label{Thm:ENN}
\end{Thm}

\begin{proof}
First let us describe the construction of the homomorphism $\CT_h^{t,u}$. To this end define a homomorphism 
$H \colon \Tau \times [0,1]^2 \to \Rr^{N_1 + N_2}$ by setting
\[
H(\gamma_1, \gamma_2, t, u) = (t h_1(\gamma_1), u h_2(\gamma_2)),
\]

as well as the groupoid
\[
\Tau_H := \Rr^{N_1 + N_2} \rtimes_H (\Tau \times [0,1]^2) \rightrightarrows \Rr^{N_1 + N_2} \times \Tau^{(0)}.
\]
This is the semi-direct product groupoid of the action induced by the homomorphism $H$. We set
\[
(\Tau_H)^{t,u} := \Tau_{H|_{M_1 \times M_2 \times \{t\} \times \{u\} \times \Rr^{N_1 + N_2}}}.
\]

Hence 
\[
(\Tau_H)^{0,0} = \Tau \times \Rr^{N_1 + N_2}, \ (\Tau_H)^{1,1} = \Tau_h, \ (\Tau_H)_{|(0,1] \times (0,1]} \cong \Tau_h \times (0,1] \times (0,1],
\]
so that 
\[
C^{\ast}((\Tau_H)_{(0,1] \times (0,1]})) \cong C^{\ast}(\Tau_h \times (0,1]^2).
\]
We obtain the short exact sequence:
\[
\xymatrix{
C^{\ast}(\Tau_h \times (0,1]^2) \ar@{>->}[r] & C^{\ast}(\Tau_H) \ar@{->>}[r]^-{e_{00}} & C^{\ast}(\Tau \times \Rr^{N_1 + N_2}).
}
\]

Since the kernel is contractible, we have the deformation index map 
\[
\D_h \colon K_{\ast}(C^{\ast}(\Tau \times \Rr^{N_1 + N_2})) \to K_{\ast}(C^{\ast}(\Tau_h)).
\]

The natural map $\Tau_H \to [0,1]^2$ gives rise to a continuous field of groupoids over $[0,1]^2$ and since $\Tau$ is amenable by Theorem \ref{Thm:doubleTG}, we get that $C^{\ast}(\Tau_H)$ is the space of continuous sections of a continuous field of $C^{\ast}$-algebras. The deformation index, as specified above, coincides with the morphism in \cite{enn}[Thm. 3.1]. 

Furthermore, since $N_1 + N_2$ is even, denote the Bott isomorphism by 
\[
B \colon K_{\ast}(C^{\ast}(\Tau)) \iso K_{\ast}(C^{\ast}(\Tau \times \Rr^{N_1 + N_2})).
\]

Setting $\CT_h := \D_h \circ B$ we have constructed the homomorphism $\CT_h$. 
By fixing a Haar system $(\mu_x^{t,u})_{(x,t,u) \in M_1 \times M_2 \times [0,1]^2}$ on the groupoid $\Tau$ we obtain an induced Haar system $(\mu_{z}^h)_{z \in (\Tau_h)^{(0)}}$ on $\Tau_h$. The additive group $\Rr^{N_1 + N_2}$ acts on 
$C^{\ast}(\Tau)$ by automorphisms via the formula $\alpha_v(f)(\gamma) = e^{i \scal{v}{h(\gamma)}} f(\gamma)$ for $f \in C_c^{\infty}(\Tau)$, see \cite{connes}[Prop. II.5.7] for the details. In case $N_1 + N_2$ is even, we have the Connes-Thom isomorphism in $K$-theory \cite{connes}[II.C] 
\[
C^{\ast}(\Tau_h) \cong \Rr^{N_1 + N_2} \rtimes_{\alpha} C^{\ast}(\Tau).
\] 
Altogether, we obtain that the Connes-Thom map takes the form:
\[
\CT_h \colon K_{\ast}(C^{\ast}(\Tau)) \to K_{\ast}(C^{\ast}(\Rr^{N_1 + N_2} \rtimes_h \Tau)) \cong \Rr^{N_1 + N_2} \rtimes_{\alpha} K_{\ast}(C^{\ast}(\Tau)).
\]
 
Via Lemma \ref{Lem:Tauh} we have $\Tau_h \sim_{\M} \B^{t,u}$, hence
\[
\CT_h^{t=0, u=0} \colon K_c^{\ast}(T(M_1 \times M_2)) \iso K_c^{\ast}(\N).
\]

By virtue of \cite{enn}[Thm. 5.1] $\CT_h$ as constructed, coincides with the usual Connes-Thom isomorphism. 
\end{proof}

\begin{Rem}
Connes defines in \cite[II.10]{connes} the geometric $K$-homology $K_{\ast}^{geo}(\G)$ for a proper Lie groupoid $\G$. Denote by
$\C_{\G}$ the category of proper right $\G$-spaces and homotopy classes of $\G$-equivariant maps. Given $Z \in \C_{\G}$, we denote
by $q \colon Z \to \Gop$ the anchor map of the action. We define $T^q Z := \ker dq$, the vertical tangent bundle.
The definition of the geometric $K$-homology relies on the following generators and relations:
\begin{itemize}
\item $(Z, x)$ where $Z \in \C_{\G}$, $x \in K_{\ast}(C^{\ast}(T^q Z \rtimes \G))$. 

\item $(Z_1, x_1) \sim (Z_2, x_2) :\Leftrightarrow$ there is a $Z \in \C_{\G}$ and $\G$-equivariant smooth maps 
$f_j \colon Z_j \to Z$ such that $(df_{1})_{!}(x_1) = (df_{2})_{!}(x_2)$. The pushforward is defined via $E$-theory, cf.
\cite{connes}[II.10.$\alpha$] for the details. 
\end{itemize}

The geometric assembly map $\mu^{geo} \colon K_{\ast}^{geo}(\G) \to K_{\ast}(C^{\ast}(\G))$ is defined as 
$\mu^{geo}([Z, x]_{\sim}) = (\pi_Z)_{!}(x)$, where $\pi_Z \colon Z \to \mathrm{E\G}$ is the natural projection map 
to the classifying space, i.e. final object $\mathrm{E\G} \in \C_{\G}$. The groupoid $\G$ has the \emph{geometric Baum-Connes property}
if $\mu^{geo}$ is an isomorphism. Recall the following result due to Connes: If $Z \in \C_{\G}$ is a final object, then $K_{\ast}^{geo}(\G) = K_{\ast}(C^{\ast}(T^q Z \rtimes \G))$. 
We apply this to our situation: Given $(\G, h)$ an amenable Lie groupoid with free and proper action of $\G$ on $\Rr^N$ induced 
by $h$. Then set $\G_h := \Rr^{N} \rtimes_h \G \rightrightarrows \Gop \times \Rr^{N}$ and denote by $\B_h$ the orbit space. Consider now $\G$ to be either the tangent groupoid $\G_M^t$ or the double deformation groupoid $\Tau$ with $h$ being the appropriate homomorphism $h^t$ according to \cite{connes}[p.105] or $h^{t,u}$ as defined at the beginning of this section. By \cite{connes}[10.$\alpha$, Prop. 5] we have $K_{\ast}^{geo}(\G) = K_c^0(\B_h)$.
We obtain the diagram:
\begin{figure}[H]
\begin{tikzcd}[row sep=huge, column sep=huge, text height=1.5ex, text depth=0.25ex]
\arrow[bend right=90]{d}{\CT_h} K_{\ast}(C^{\ast}(\G)) & \\
K_{\ast}^{geo}(\G) \cong K_{\ast}(C^{\ast}(\G_h)) \arrow{u}{\mu^{geo}} \arrow[r, "\simeq"'] & \arrow{ul} K_c^{0}(\B_h) 
\end{tikzcd} 
\end{figure}

By Theorem \ref{Thm:ENN} it follows that the double tangent groupoid $\Tau$ has the geometric Baum-Connes property. Then by the considerations of \cite{l}[p. 98] and the construction of $\CT_h$ in the proof of Theorem \ref{Thm:ENN} we obtain that $(\mu^{geo})^{-1} = \CT_h$. The computation of the index
invariants can be reduced to the computations on the orbit space $\B_h$. This is the basic idea to Connes' proof of the Atiyah-Singer 
index theorem, when $\G$ is the tangent groupoid of a compact manifold without boundary. 
\label{Rem:geomred}
\end{Rem}

\begin{Def}
The $\otimes$-topological index map is defined as the map $\ind_{\mathrm{top}}^{\otimes} \colon K_c^{0}(T^{\ast}(M_1 \times M_2)) \to \Zz$ via the
composition $\ind_{top}^{\otimes} := \beta \circ (e_{11}^{h})_{\ast} \circ (e_{00}^{h})_{\ast}^{-1} \circ \tau$, where $\tau \colon K_c^{0}(T^{\ast}(M_1 \times M_2)) \to K_c^{0}(\N_1 \times \N_2)$ 
is the Thom-isomorphism, $\beta$ the Bott periodicity isomorphism and $e_{00}^h$, $e_{11}^{h}$ denote the restrictions to the $t=0,u=0$ and $t=1,u=1$ components of $\B^{t,u}$.
\label{Def:indtop}
\end{Def}

\begin{Thm}
We have the equality of index maps $\ind_a^{\otimes} = \ind_{\mathrm{top}}^{\otimes}$.
\label{Thm:indtop}
\end{Thm}

\begin{proof}
By the previous discussion we have the following commuting diagram
\begin{figure}[H]
\begin{tikzcd}[row sep=huge, column sep=huge, text height=1.5ex, text depth=0.25ex]
\Zz \arrow[r, equal] & \Zz \\
K_0(\K) \arrow[u, "\M", "\simeq"'] \arrow{r}{\CT_h^{t=1, u=1}} & K_c^0(\Rr^{N_1 + N_2}) \arrow[u, "\beta", "\simeq"'] \\
K_0(C^{\ast}(\Tau^{t,u}) \arrow[d, "(e_{00})_{\ast}", "\simeq"'] \arrow{u}{(e_{11})_{\ast}} \arrow[r, "\CT_h", "\simeq"'] & K_0(C^{\ast}(\Tau_h)) \cong K_c^0(\B^{t,u}) \arrow{u}{(e_{11}^h)_{\ast}} \arrow[d, "(e_{00}^h)_{\ast}", "\simeq"'] \\
\arrow[bend left=60]{uuu}{\ind_{a}^{\otimes}} K_c^0(T^{\ast}(M_1 \times M_2)) \arrow[r, "\CT_h^{t=0,u=0}", "\simeq"'] & K_c^0(\N_1 \times \N_2).
\end{tikzcd} 
\end{figure}

A computation shows that
\begin{align*}
&\ind_{\mathrm{top}}^{\otimes} = \M \circ (\CT_h^{t=0, u=0})^{-1} \circ (e_{11}^h)_{\ast} \circ (e_{00}^h)_{\ast}^{-1} \circ \CT_h^{t=0,u=0} \\
&= \M \circ (\CT_h^{t=0,u=0})^{-1} \circ (\CT_h^{t=0,u=0}) \circ (e_{11})_{\ast} \circ \CT_h^{-1} \circ \CT_h  \\
&\circ (e_{00})_{\ast}^{-1} \circ (\CT_h^{t=0,u=0})^{-1} \circ \CT_h^{t=0,u=0} \\
&= \M \circ (e_{11})_{\ast} \circ (e_{00})_{\ast}^{-1} \\
&= \ind_a^{\otimes}.
\end{align*} 
\end{proof}

We refer to \cite{dln} for a comparison of the topological index, as defined via the tangent groupoid construction and the topological index, as originally defined by Atiyah-Singer. It can be shown that these indices in fact agree. With an additional computation one can obtain a cohomological index formula. 

\section*{Acknowledgements}
I thank the referees for very helpful remarks that lead to significant improvements. For useful discussions I am thankful to Bernd Ammann, Ulrich Bunke, Alexander Engel, Jean-Marie Lescure and Elmar Schrohe.  
This research was supported by the DFG-SPP 2026 `Geometry at Infinity' program.

\end{document}